\documentclass[1 [leqno,11pt]{amsart}

\usepackage{amsmath}
\usepackage{amssymb}
\usepackage{amsfonts}
\usepackage{mathrsfs}
\usepackage{shortvrb,psfrag}
\usepackage{enumerate}
\usepackage{color}
\usepackage[dvips]{graphicx}

\let\f=\frac

\let\S=\Sigma
\let\Om=\Omega

\def\ba{\begin{eqnarray}}
\def\ea{\end{eqnarray}}

\def\R{\Bbb R}
\def\S{\Bbb S}
\def\N{\Bbb N}
\def\Z{\Bbb Z}

\def\no{\noindent}
\def\na{\nabla}

\def\eqdef{\buildrel\hbox{\footnotesize def}\over =}
\def\endproof{\hphantom{MM}\hfill\llap{$\square$}\goodbreak}

\newcommand{\beq}{\begin{equation}}
\newcommand{\eeq}{\end{equation}}
\newcommand{\ben}{\begin{eqnarray}}
\newcommand{\een}{\end{eqnarray}}
\newcommand{\beno}{\begin{eqnarray*}}
\newcommand{\eeno}{\end{eqnarray*}}

\newtheorem{Theorem}{Theorem}[section]

\newtheorem{Proposition}[Theorem]{Proposition}
\newtheorem{Lemma}[Theorem]{Lemma}

\newtheorem{Remark}[Theorem]{Remark}

\begin{document}

\title[Energy identity for approximate harmonic map]{Energy identity for approximate harmonic maps from surface to general targets}

\author{Wendong Wang}
\address{School of Mathematical Sciences\\ Dalian University of Technology\\ Dalian, 116024,  China}
\email{wendong@dlut.edu.cn}

\author{Dongyi Wei}
\address{School of Mathematical Sciences\\ Peking University\\ Beijing 100871,China}
\email{jnwdyi@163.com}

\author{Zhifei Zhang}
\address{School of Mathematical Sciences and LMAM \\ Peking University\\ Beijing 100871,China}
\email{zfzhang@math.pku.edu.cn}

\date{\today}

\maketitle

\begin{abstract}
Let $u_n$ be a sequence of mappings  from a closed Riemannian surface  $M$  to a general Riemannian manifold $N$. If $u_n$ satisfies
\beno
\sup_{n}\big(\|\nabla u_n\|_{L^2(M)}+\|\tau(u_n)\|_{L^{p}(M)}\big)\leq \Lambda\quad \text{for some}\,\,p>1,
\eeno
where $\tau(u_n)$ is the tension field of $u_n$, then there hold the so called energy identity and neckless property during blowing  up. This result is sharp by Parker's example, where the tension fields of the mappings from Riemannian surface are bounded in $L^1(M)$ but the energy identity fails.
\end{abstract}

\setcounter{equation}{0}
\section{Introduction}
Let $(M,g)$ be a closed Riemannian manifold and $(N,h)$ be a Riemannian manifold without boundary. Let $u$ be a mapping from $M$ to $N$ in $W^{1,2}(M,N)$. We define the Dirichlet energy of $u$ as follows
\beno
E(u)=\int_{M}e(u)dV,
\eeno
where $dV$ is the volume element of $(M,g)$, and $e(u)$ is the density of $u$
\beno
e(u)=\frac{1}{2}|du|^2=\text{Trace}_g u^{*}h,
\eeno
where $u^{*}h$ is the pull-back of the metric tensor $h$.

A map $u\in C^1(M,N)$ is called harmonic if it is a critical point of the energy $E$. By the Nash embedding theorem, $(N,h)$ can be isometrically embedded into a Euclidean space $\R^k$ for some positive integer $k$ with the metric induced from the Euclidean metric. Hence,
a map $u\in C^1(M,N)$ can be viewed as a map of $C^1(M,\R^k)$ whose image lies in $N$. Then we can obtain the Euler-Lagrange equation
\ben\label{eq:u}
\triangle u-A(u)(du,du)=0,\ \ \text{or} \ \ P(u)\triangle u=0,
\een
where $A(u)(du,du)$ is the second fundamental form of $N$ in $\R^k$. Let $P(y):\R^k\rightarrow T_yN$ be the orthogonal projection map. The tension field $\tau(u)$ is defined by
\ben\label{eq:u-app}
\tau(u)\eqdef\triangle u-A(u)(du,du)=P(u)\triangle u.
\een
Then $u$ is harmonic if and only if  $\tau(u)=0$. We refer to \cite{LW-book} for the systematic study on the harmonic maps.

The harmonic maps are of special interest when $M$ is a Riemannian surface, because the Dirichlet energy is conformally invariant in two dimensions. It is an important question to understand the limiting behavior of sequences of harmonic maps. Let $u_n$ be a sequence of mappings  from Riemannian surface $M$ to $N$ with bounded
energy. It is clear that $u_n$ converges weakly to $u$ in $W^{1,2}(M,N)$ for some $u\in W^{1,2}(M,N)$. In general, it may not converge strongly in $W^{1,2}(M,N)$ due to the concentration of the energy at finitely many points \cite{SU}. Thus, it is natural to ask (1) whether the lost
energy is exactly the sum of energies of some harmonic spheres(bubbles), which are defined as harmonic maps from $\S^2$ to $N$; (2) whether attaching all possible bubbles to the weak limit gives uniform convergence. The first one is so called the energy identity, and the second one is called the bubble tree convergence.

When $\tau(u_n)=0$,  Jost and Parker \cite{Parker} independently  proved the energy identity and neckless property during blowing up.
When $\tau(u_n)$ is bounded in $L^2(M)$, the energy identity was proved by Qing \cite{Qing} for the sphere, by Ding and Tian \cite{DT} and Wang \cite{Wang} for general target manifold. Qing and Tian \cite{QT} also proved neckless property during blowing up. One can refer to\cite{Top1,Top2, LW-CVPDE} for the related results of the heat flow of harmonic maps. Notice that $L^2(M)$ space for the tension field is not conformally invariant. So, a natural substitution of $L^2(M)$ space seems $L^1(M)$ space. However, Parker \cite{Parker} construct a sequence of mappings from Riemannian surface, in which the tension fields are bounded in $L^1(M)$ but the energy identity fails. This motivates  the following important question:
\medskip

{\bf Q: whether the energy identity
or the neckless property holds for a general target
manifold when the tension field is bounded in $L^p(M)$ for $p>1$?}
\medskip

This question has been solved by Lin and Wang \cite{LW-CAG} when the target manifold is the sphere. Li and Zhu \cite{LZ1} also proved the energy identity for the tension fields bounded in $L\ln^+L$, and constructed a
sequence of mappings with tension fields bounded in $L\ln^+L$ so that
there is a positive neck during blowing up. For general target manifolds, partial important progress has been made: Li and Zhu \cite{LZ2} proved the energy identity and  the neckless property for $p\geq \frac65$; Luo \cite{Luo}  obtained the same result under the following condition
\beno
\Big(\int_{D_r\setminus D_{r/2}}|\tau(u_n)|^2dx\Big)^\f12\le Cr^{-a}
\eeno
for some $a\in (0,1)$ and any $r\in (0,1)$. Here we denote by $D(x,r)$ the ball with the center $x$ and the radius $r$ and $D_r=D(0,r)$.
\medskip

The goal of this paper is to give a positive answer to {\bf Q}.
When $\tau(u_n)$ is bounded in $L^p(M)$ for some $p>1$, the small energy regularity (see Lemma \ref{lem:small energy}) implies that $u_n$ converges strongly in $W^{1,2}(M,N)$ outside a finite set of points. For simplicity,  we assume that $M$ is the unit disk $D_1=D(0,1)$ and 0 is the only one singular point.\medskip

Our main result is stated as follows.

\begin{Theorem}\label{thm:energy}
Let $\{u_n\}$ be a sequence of mappings from $D_1$ to $N$ in $W^{1,2}(D_1,N)$ with tension field $\tau(u_n)$ satisfying
\beno
\|u_n\|_{\dot{W}^{1,2}(D_1)}+\|\tau(u_n)\|_{L^p(D_1)}\leq \Lambda
\eeno
for some $p>1$, and for $0<\delta<1$,
\beno
u_n\rightarrow u\,\,{\rm strongly}\,\,{in}\,\, W^{1,2}(D_1\backslash{D_{\delta}},\R^k) \quad {\rm as\,\,} n\rightarrow\infty.
\eeno
Then there exist a subsequence of $\{u_n\}$ (still denoted by $\{u_n\}$) and some nonnegative integer $k_0$ such that for any $i=1,\cdots,k_0$, there exist points $x_n^i$, positive numbers $r_n^i$ and a nonconstant harmonic sphere $w_i$ (a map from $\R^2\cup\{\infty\}\rightarrow N$),
which satisfy

\begin{itemize}

\item[1.]  for any $i=1,\cdots,k_0$, $x_n^i\rightarrow 0, r_n^i\rightarrow 0$ as $n\rightarrow\infty.$

\item[2.]
\beno
\lim_{n\rightarrow\infty}\left(\frac{r_n^i}{r_n^j}+\frac{r_n^j}{r_n^i}+\frac{|x_n^i-x_n^j|}{r_n^i+r_n^j}\right)=\infty\quad {\rm for \,\, any}\quad i\neq j.
\eeno

\item[3. ]$w^i$ is the weak limit or strong limit of $u_n(x_n^i+r_n^ix)$ in $W^{1,2}_{loc}(\R^2,N)$.\\
\item[4.] {\bf Energy identity}
\beno
\lim_{n\rightarrow\infty}E(u_n,D_1)=E(u,D_1)+\sum_{i=1}^{k_0}E(w^i).
\eeno

\item[5.] {\bf Neckless property:} The image $u(D_1)\cup \bigcup_{i=1}^{k_0}w^i(\R^2)$ is a connected set.

\end{itemize}
\end{Theorem}

\setcounter{equation}{0}
\section{Bubble tree structure of approximate harmonic maps}

Let us first recall the following small energy regularity result \cite{DT, LZ1}.

\begin{Lemma}\label{lem:small energy}
Let $u$ be a mapping from $D_1$ to $N$ in $W^{1,2}(D_1,N)$ with tension field $\tau(u)\in L^p(D_1)$  for $p>1.$ Then there exists a positive constant
$\epsilon_N$ depending on the target manifold $N$ such that if $E(u,D_1)\leq \epsilon_N^2$, then
\beno
\|u-\bar{u}\|_{W^{2,p}(D_{\frac12})}\leq C \left(\|\nabla u\|_{L^2(D_1)}+\|\tau(u)\|_{L^p(D_1)}\right),
\eeno
where $\bar{u}$ is the mean value of $u$ on the disk $D_{\frac12}$.
\end{Lemma}

Let $u_n$ be a sequence of mapping from $D_1$ to $N$ satisfying
\beno
\|u_n\|_{W^{1,2}(D_1)}+\|\tau(u_n)\|_{L^p(D_1)}\leq \Lambda
\eeno
for some $p>1$.
A point $x\in D_1$ is called an energy concentration point (blow-up point)
of $u_n$ if
for any $r$ so that $D(x,r)\subset D_1$, we have
\beno
\limsup_{n\rightarrow\infty}E(u_n,D(x,r))>\epsilon_N^2.
\eeno

Based on Lemma \ref{lem:small energy}, using standard blow-up argument as in \cite{Qing, DT}, it can be proved that for fixed sufficiently small $\epsilon\in(0,\epsilon_N)$, there exists $k_0\ge 0$ so that for any $i=1,\cdots, k_0$, there exist a point $x_n^i$
, a positive number $r_n^i$,
and a nonconstant harmonic sphere $w^i$ satisfying the conclusions $1-3$ in Theorem \ref{thm:energy}. Moreover, it holds that
\begin{itemize}

\item[1.] $\ w^i$ is the strong limit of $u_n(x_n^i+r_n^ix)$ in $W^{1,2}_{loc}(\R^2\setminus Z_i,N)$, where $Z_i$ is the set of blow-up points of this scaling sequence, thus $Z_i$ is finite and for $x\in Z_i$,
\beno
m_i(x)\triangleq\lim\limits_{\delta\to0}\lim\limits_{n\to\infty}E(u_n,D(x_n^i+r_n^ix,\delta r_n^i))\geq \epsilon_N.
\eeno
For the sake of completeness, we denote $x_n^0=0,\ r_n^0=1,\ Z_0=\{0\},\ w^0=u.$

\item[2.] $\ \exists\ f:\{1,\cdots,k_0\}\rightarrow \mathbb{Z}$ and $\delta_0, R_0>0$ so that
\beno
&&\ 0\leq f(j)<j,\quad \lim\limits_{n\to\infty}\dfrac{r_n^j}{r_n^{f(j)}}=0,\quad\lim\limits_{n\to\infty}
\dfrac{x_n^j-x_n^{f(j)}}{r_n^{f(j)}}=y_j\in Z_{f(j)},\\
&& \text{and} \quad E\big(u_n,D(x_n^j,r_n^{f(j)}{\delta_0})\backslash D(x_n^j,r_n^{j}{R_0})\big)\leq\epsilon.
\eeno

\item[3.] If $f(i)=f(j)$ and $y_i=y_j,$ then $i=j$, $Z_i=\big\{y_j|f(j)=i\big\}.$
\end{itemize}

Let us just present a sketch for the construction of $(x_n^1,\ r_n^1)$(see P.118-121 in \cite{Mcd} for similar construction). As $m_0(0)=\lim\limits_{\delta\to0}\lim\limits_{n\to\infty}E(u_n,D_{\delta})\geq \epsilon_N$, there exists  $\delta>0$ so that
\beno
\ |E(u_n,D_{\delta}) -m_0(0)|\leq\dfrac{\epsilon}{4}
\eeno
for $n$ sufficiently large. Let $Q_n(t)\triangleq\sup\limits_{D(z,t)\subseteq D_{\delta}}E(u_n,D(z,t))$ for $0\leq t\leq {\delta}$. Then $Q_n(t)$
is continuous and non-decreasing in $t, Q_n(0)=0$ and $Q_n(\delta)=E(u_n,D_{\delta})$.
Therefore, there exists $0<r_n^1<\delta$ such that $Q_n(r_n^1)=\max\big(Q_n(\delta)-\epsilon,\frac{\epsilon}{2}\big)$ and there exists $D(x_n^1, r_n^1)\subseteq D_{\delta}$ such that $E(u_n,D(x_n^1, r_n^1))=Q_n(r_n^1)$. Thus, we have  $x_n^1\to0,\ r_n^1\to0$ and
\begin{align*}
E\big(u_n,D(x_n^1,{\delta})\backslash D(x_n^1,r_n^{1})\big)\leq& E(u_n,D_{\delta})- E(u_n,D(x_n^1,r_n^{1}))\\
=&Q_n(\delta)-Q_n(r_n^1)\leq\epsilon
\end{align*}
for $n$ sufficiently large. Hence, we can take $f(1)=0$.
Moreover, by Lemma \ref{lem:small energy}, $u_n(x_n^1+r_n^1x)$ has a subsequence, which  strongly converges in $W^{1,2}_{loc}(\R^2\setminus Z_1,N)$ to $ w^1,$ and
for $x\in Z_1$,
\beno
m_1(x)\leq \limsup\limits_{n\to\infty}Q_n(r_n^1)\leq \max\big(m_0(0)-\frac{3}{4}\epsilon,\frac{\epsilon}{2}\big),
\eeno
which implies that this construction can only happen finite times.

\begin{Remark}
In fact, Zhu \cite{Zhu2} proved the bubble tree theorem under the weaker condition
\beno
\|u_n\|_{W^{1,2}(D_1)}+\|\tau(u_n)\|_{L\ln^+L(D_1)}\leq \Lambda,
\eeno
where $\|f\|_{L\ln^+L(D_1)}\eqdef\int_{D_1}|f(x)|\ln(2+|f(x)|)dx$.
\end{Remark}

Now, by Property 1 and Lemma \ref{lem:small energy}, for fixed $0<\delta<1<R$ and $1\leq i\leq k_0$, we have
\beno
&& u_n(x_n^i+r_n^ix)\to w^i \ \text{strongly in}\ W^{1,2}(D_R\setminus\cup_{x\in Z_i}D(x,\delta))\cap C^0(D_R\setminus\cup_{x\in Z_i}D(x,\delta)),\\
&& u_n\to u\ \text{strongly in}\ W^{1,2}(D_1\setminus D_{\delta})\cap C^0(D_{\f 12}\setminus D_{\delta}),
\eeno
as $n\to \infty$. Therefore, as $n\to \infty$, \beno
&&E(u_n,D(x_n^i,r_n^iR)\setminus\cup_{x\in Z_i}D(x_n^i+r_n^ix,r_n^i\delta))\to E(w^i,D_R\setminus\cup_{x\in Z_i}D(x,\delta)),\\
&&E(u_n,D_1\setminus D_{\delta})\to E(u,D_1\setminus D_{\delta}).
\eeno
Moreover,
$$\limsup\limits_{n\rightarrow \infty}
\textrm{osc}\ \big(u_n, D(x^i_n,r_n^{f(i)}\delta)\setminus D(x^i_n,r_n^{i}R)\big)\geq \textrm{diam}(w^{f(i)}(\partial D_{\delta})\cup w^{i}(\partial D_{R})),$$
and for $n$ sufficiently large,
\beno
&&\Big|E(u_n,D_1)-E(u_n,D_1\setminus D_{\delta})-\sum\limits_{i=1}^{k_0}E(u_n,D(x_n^i,r_n^iR)\setminus\cup_{x\in Z_i}D(x_n^i+r_n^ix,r_n^i\delta))\Big|\\
&&\leq\sum\limits_{i=1}^{k_0}E\big(u_n, D(x^i_n,r_n^{f(i)}\delta)\setminus D(x^i_n,r_n^{i}R)\big).
\eeno
Thus, the energy identity is equivalent to show that  there is no energy on the neck during blow-up process, i.e.,
\ben\label{eq:energy-lim}
\lim_{\delta\rightarrow 0}\lim_{R\rightarrow \infty}\limsup_{n\rightarrow \infty}
E\big(u_n, D(x^i_n,r_n^{f(i)}\delta)\setminus D(x^i_n,r_n^iR)\big)=0\quad
\text{for } i=1,\cdots k_0.
\een
While, the neckless property is equivalent to show  that there is no osicillation on the neck, i.e.,
\ben\label{eq:neckless}
\lim_{\delta\rightarrow 0}\lim_{R\rightarrow \infty}\limsup_{n\rightarrow \infty}
\textrm{osc}\big(u_n, D(x^i_n,r_n^{f(i)}\delta)\setminus D(x^i_n,r_n^{i}R)\big)=0\quad
\text{for } i=1,\cdots k_0.
\een

In order to prove (\ref{eq:energy-lim}) and (\ref{eq:neckless}), our key idea is to show that the Hopf differential of the approximate harmonic map $u$ can be approximated by a holomorphic function, where the error is quantized by the tension field of $u$. The result is trivial in the case when $\tau(u)=0$, because the Hopf differential of harmonic map $u$ is holomorphic.

\setcounter{equation}{0}
\section{The Coulomb gauge frame}

Consider $\Omega=D_1$ or $D_1\backslash D_{r}$ for $0<r\leq \frac{1}{4}$.  Assume that
\ben
E(u,\Omega)\leq \epsilon_0,
\een
 where $\epsilon_0>0$ will be determined later.

Recall that $(N,h)$ can be isometrically embedded into $\R^k$.
Let $\overline{N}$ be a submanifold of $\R^{2k}$ defined by
\beno
\overline{N}\eqdef \big\{(y,y')\in\R^k\times\R^k: y\in N,\ y'\perp T_yN\big\}.
\eeno
Then $N=N\times\{0\}$ is a totally geodesic submanifold
of $\overline{N}. $ As in \cite{Helein, LW-book}, we may introduce the Coulomb gauge frame of $u^*T\overline{N}$. Let us present the construction of Coulomb gauge.\smallskip

The following lemma makes use of Hardy space (for example, see \cite{Helein, LW-book}).

\begin{Lemma}\label{Hardy}If $\triangle\Phi=\dfrac{\partial f}{\partial x_1}\dfrac{\partial g}{\partial x_2}-\dfrac{\partial f}{\partial x_2}\dfrac{\partial g}{\partial x_1},$ then we have
$$\left\|d\Phi\right\|_{L^{2,1}(\R^2)}\leq C\left\|df\right\|_{L^{2}(\R^2)}\left\|dg\right\|_{L^{2}(\R^2)}.$$
Here $L^{p,q}(\R^2)$ is the Lorentz space.
\end{Lemma}

We also need the following extension lemma.

\begin{Lemma}
Let $f\in W^{1,2}(\Om)$. Then
then there exists an extension $\bar f\in \dot W^{1,2}(\R^2)$ of $f$  so that
\beno
E(\bar f,\R^2)\le CE(f,\Omega).
\eeno
Here $C$ is a constant independent of $r$.
\end{Lemma}

\begin{proof}
We only consider the case $\Omega=D_1\backslash D_{r}$.  First of all, we can find $f_j\in \dot{W}^{1,2}(\R^2),\ j=1,2$ such that $E(f_j,\R^2)\leq CE(f_j,D_1\backslash D_{\frac{1}{2}})$, and $f_1(z)=f(z),\ \ f_2(z)=f(2rz)$ in $D_1\backslash D_{\frac{1}{2}}$. Then we can extend $f$ to $\dot{W}^{1,2}(\R^2)$ by taking $\bar f(z)=f_1(z)$ for $|z|>1$ and $\bar f(z)=f_2(z/2r)$ for $|z|<r$. Then we find that
\begin{align*}
E(\bar f,\R^2)&=E(\bar f,D_{2r})+E(\bar f,D_{\frac{1}{2}}\backslash D_{2r})+E(\bar f,\R^2\backslash D_{\frac{1}{2}})\\
&=E(f_2,D_{1})+E(f,D_{\frac{1}{2}}\backslash D_{2r})+E(f_1,\R^2\backslash D_{\frac{1}{2}})\\
&\leq CE(f_2,D_1\backslash D_{\frac{1}{2}})+E(f,D_{\frac{1}{2}}\backslash D_{2r})+CE(f_1,D_1\backslash D_{\frac{1}{2}})\\
&=CE(f,D_{2r}\backslash D_{r})+E(f,D_{\frac{1}{2}}\backslash D_{2r})+CE(f,D_1\backslash D_{\frac{1}{2}})\\
&\leq CE(f,D_1\backslash D_r).
\end{align*}
This completes the proof.
\end{proof}

Let  $\mathcal{A}(\Omega)$ be the set of  $R=(e_1,\cdots,e_k)\in \dot{W}^{1,2}(\R^2,\R^{2k\times k})$ such that
$$R^TR=I_k,\ RR^T=\left(\begin{array}{cc}P(u) &  \\ & I_k-P(u)  \end{array}\right)\ \ \ \ \text{in}\ \Omega.
$$

First of all,  we show that $\mathcal{A}(\Omega)$ is nonempty and $E(R)=\frac{1}{2}\int_{\R^2}|dR|^2$ attains a minimum in $\mathcal{A}(\Omega)$.

Indeed, let $R_0=\left(\begin{array}{c}P(u)   \\  I_k-P(u)  \end{array}\right)$ in $\Omega$ and extend $R_0$ to $\R^2$
 so that $\int_{\R^2}|dR_0|^2\leq C\int_{\Omega}|dR_0|^2$.
 Then $R_0\in \mathcal{A}(\Omega)$ and $E(R_0)\leq CE(u,\Omega)$. Let $R_n\in \mathcal{A}(\Omega)$ so that
 \beno
 \lim\limits_{n\to\infty}E(R_n)=E_0=\inf\limits_{R\in \mathcal{A}(\Omega)} E(R).
 \eeno
 Then there exists a subsequence of $\{R_n\}$ (still denoted by $\{R_n\}$) and $R\in\dot{W}^{1,2}(\R^2,\R^{2k\times k})$ such that $R_n\to R$ strongly in $L^2(\Omega)$ as $n\to\infty$, and $dR_n\rightharpoonup dR$ weakly in $L^2(\R^2)$ as $n\to\infty$. Then $R_n^TR_n\to R^TR$ and $R_nR_n^T\to RR^T$ strongly in $L^1(\Omega)$ as $n\to\infty$, and $E(R)\leq E_0$. Therefore, $R\in \mathcal{A}(\Omega)$ and $E(R)=E_0\leq E(R_0)\leq CE(u,\Omega)$. \smallskip

 Now, for $ \psi\in C_0^{\infty}(\R^2,so(k))$, we have $R\exp{t\psi}\in \mathcal{A}(\Omega)$ for $t\in\R$. Therefore, $E(R\exp{t\psi})\geq E_0=E(R)$ and
 \beno0=\left.\frac{d}{dt}\right|_{t=0}E(R\exp{t\psi})=\int_{\R^2}\left\langle dR,\left.\frac{d}{dt}\right|_{t=0}d(R\exp{t\psi})\right\rangle=\int_{\R^2}\left\langle dR,d(R{\psi})\right\rangle\\=\int_{\R^2}(\left\langle dR,R(d{\psi})\right\rangle+\left\langle dR,(dR){\psi}\right\rangle)=\int_{\R^2}\left\langle dR,R(d{\psi})\right\rangle=\int_{\R^2}\left\langle R^TdR,d{\psi}\right\rangle,\eeno as $\left\langle dR,(dR){\psi}\right\rangle=0$. 
 For $ \psi\in C_0^{\infty}(\R^2,R^{k\times k})$, we have $ \psi^T-\psi\in C_0^{\infty}(\R^2,so(k))$. So,
\begin{align*}
0=\int_{\R^2}\left\langle R^TdR,d({\psi^T-\psi})\right\rangle=&\int_{\R^2}\left\langle (R^TdR)^T-R^TdR,d{\psi}\right\rangle\\
=&\int_{\R^2}\left\langle (dR^T)R-R^TdR,d{\psi}\right\rangle.
\end{align*}
Therefore, $d^*((dR^T)R-R^TdR)=0$ and there exists $\Phi\in\dot{W}^{1,2}(\R^2,\R^{k\times k})$ so that $$\frac{1}{2}((dR^T)R-R^TdR)=\frac{\partial\Phi}{\partial x_2}dx_1-\frac{\partial\Phi}{\partial x_1}dx_2.$$
Noticing that
$$
\triangle\Phi=\dfrac{\partial R^T}{\partial x_1}\dfrac{\partial R}{\partial x_2}-\dfrac{\partial R^T}{\partial x_2}\dfrac{\partial R}{\partial x_1},
$$
we infer from Lemma \ref{Hardy}  that
$$\left\|d\Phi\right\|_{L^{2,1}(\R^2)}\leq C\left\|dR^T\right\|_{L^{2}(\R^2)}\left\|dR\right\|_{L^{2}(\R^2)}= C\left\|dR\right\|_{L^{2}(\R^2)}^2\leq CE(u,\Omega).$$
Thanks to $0=dI_k=d(R^TR)=(dR^T)R+R^TdR$  in $\Omega$, we have
\beno
&&(dR^T)R=\frac{\partial\Phi}{\partial x_2}dx_1-\frac{\partial\Phi}{\partial x_1}dx_2,\ d^*((dR^T)R)=0\ \text{in} \ \Omega,\\
&&\left\|(dR^T)R\right\|_{L^{2,1}(\Omega)}=\left\|d\Phi\right\|_{L^{2,1}(\Omega)}\leq CE(u,\Omega).
\eeno

We introduce
\beno
A=R^T\left(\begin{array}{c}\frac{\partial u}{\partial z}   \\  0  \end{array}\right),\quad \tau^{1}=\f 14R^T\left(\begin{array}{c}\tau   \\  0  \end{array}\right),\ w=\frac{\partial R^T  }{\partial \bar{z}}R\in so(k)\otimes \mathbb{C}\ \ \text{for}\ z\in\Omega.
\eeno
 Then the system of (\ref{eq:u}) is equivalent to
 \ben
 \frac{\partial A}{\partial \bar{z}}=wA+\tau^{1},
 \een
 where $w$ satisfies
\beno
\|w\|_{L^{2,1}(\R^2)}\leq C\epsilon_0.
\eeno

Let us introduce a linear operator $T:L^{\infty}(\mathbb{C}, \mathbb{C}^{k\times k})\rightarrow L^{\infty}(\mathbb{C}, \mathbb{C}^{k\times k})$ defined by
\beno
T(B)(z)=\Big((\frac{1}{\pi z})*(wB)\Big)(z)=\int_{\mathbb{C}}\frac{w(\zeta)B(\zeta)}{\pi(z-\zeta)}d\zeta.
\eeno
Thanks to  $\frac{1}{\pi z}\in L^{2,\infty}(\mathbb{C})$ and $w\in L^{2,1}(\mathbb{C})$, we deduce that
$T$ maps $L^{\infty}(\mathbb{C}, \mathbb{C}^{k\times k})$ continuously to itself with the bound
\ben\label{eq:T-bound}
\|T\|\leq C\|w\|_{L^{2,1}(\R^2)}\le C\|\na u\|_{L^{2}(\Omega)}^{2}.
\een
Moreover,  it holds that
\ben
\frac{\partial }{\partial \bar{z}}T(B)=wB.
\een

\begin{Lemma}\label{lem:B}
If $\|T\|\le \f13$, then there exists a nonsingular matrix $B\in L^{\infty}(\mathbb{C}, \mathbb{C}^{k\times k})$ so that
\beno
B-T(B)=I_k,\quad B^TB=I_k.
\eeno
Here $I_k$ is the $k\times k$ identity matrix.
\end{Lemma}

\begin{proof}

Due to $\|T\|\le \f13$, we get by the fixed point theorem that
\beno
B-T(B)=I_k
\eeno
has a unique solution $B\in L^{\infty}(\mathbb{C}, \mathbb{C}^{k\times k})$  with
\beno
\|B-I_k\|_{L^\infty(\mathbb{C})}\leq \frac32\|T\|\le\frac12.
\eeno
Recall that $w^T=-w$ and $\frac{\partial }{\partial \bar{z}}B=\frac{\partial }{\partial \bar{z}}T(B)=wB$. Thus, we get
\beno
\frac{\partial}{\partial \bar{z}}B^T=B^Tw^T,
\eeno
which implies that
\beno
\frac{\partial}{\partial \bar{z}}(B^TB)=B^T(w^T+w)B=0.
\eeno
That means that $B^TB$ is holomorphic. On the other hand, $\lim\limits_{z\rightarrow\infty}B=I_k$. Then
\beno
B^TB=I_k.
\eeno
The proof is finished.
\end{proof}

\setcounter{equation}{0}
\section{Holomorphic approximation of Hopf differential}

Throughout this section, let us assume that $u$ is a mapping from $D_1$ to $N$ in $W^{1,2}(D_1,N)$ with the tension field $\tau(u)\in L^p(D_1)$ for some $p\in(1,2)$.  We denote by $h(z)$ the Hopf differential of $u$, i.e.,
\beno
h(z)\eqdef \big\langle\frac{\partial u}{\partial z},\frac{\partial u}{\partial z}\big\rangle=\sum_{j=1}^k\frac{\partial u^j}{\partial z}\frac{\partial u^j}{\partial z}.
\eeno

It was well-known that if $u$ is harmonic, then $h(z)$ is holomorphic.
In this section, we will show that
in general case, $h(z)$ can be approximated by a holomorphic function,
where the error is quantized by $\|\tau(u)\|_{L^p(D_1)}$. 
This result may be independent of interest.

\begin{Proposition}\label{prop:holomorphic approx}
Assume that $E(u,D_1)\leq m\epsilon$ for some $m\in \N$, where $\epsilon$ is the minimum of $\epsilon_0$ given by Lemma \ref{lem:app-small} and Lemma \ref{lem:app-small2}.
Then there exist $C_m$ and a holomorphic function $h_0$ in $D_{1/4}$ such that
\beno
\|h-h_0\|_{L^1(D_{1/4})}\leq C_m\|\tau(u)\|_{L^p(D_{1})}^{3^{1-m}}.
\eeno
\end{Proposition}

The proof of Proposition \ref{prop:holomorphic approx} is based on the following lemmas.

\begin{Lemma}\label{lem:app-small}
There exists $\epsilon_0>0$ such that if $E(u,\Omega)\leq \epsilon_0$, $\Omega=D_1$ or $D_1\setminus D_r$ for some $0<r\leq\f 14$,
then there exists a holomorphic function $h_0$ in $\Omega$ so that
\beno
\|h-h_0\|_{L^1(\Omega)}\leq C\|\tau(u)\|_{L^p(\Omega)}.
\eeno
\end{Lemma}

\begin{proof}
Thanks to $E(u,\Omega)\leq \epsilon_0$, by (\ref{eq:T-bound}) we can take $\epsilon_0$ small enough so that
$\|T\|\le C\epsilon_0^2\le \f13$. Then by Lemma \ref{lem:B}, there exists a nonsingular matrix $B\in L^{\infty}(\mathbb{C}, \mathbb{C}^{k\times k})$ so that
\beno
B-T(B)=I_k,\quad B^TB=I_k,\quad \|B-I_k\|_{L^\infty(\mathbb{C})}\leq \frac32\|T\|\le C\epsilon_0^2.
\eeno
Let $A=BG$. Then $\frac{\partial}{\partial \bar{z}}G=B^{-1}\tau^{1}$. We write $G=G_1+G_2$ with $G_2=\frac{1}{\pi z}*(B^{-1}\tau^{1})$.
Then it holds that
\beno
\frac{\partial G_2}{\partial \bar{z}}=B^{-1}\tau^1,\quad\frac{\partial G_1}{\partial \bar{z}}=0\quad\text{in } \Omega.
\eeno
Moreover,  it holds that
\beno
\|G_2\|_{L^{\frac{2p}{2-p}}(\R^2)}\leq C\|\tau^1\|_{L^{p}(\R^2)}\leq C\|\tau(u)\|_{L^{p}(\Omega)}.
\eeno
Let $h_1(z)=G_1(z)^TG_1(z)$, which is holomorphic in $\Omega$. Notice that
\beno
h(z)=A^TA=G^TB^TBG=G^TG.
\eeno
Hence, by $G=B^{-1}A$ and $\|A\|_{L^2(\Omega)}\leq C\|\nabla u\|_{L^2(\Omega)}$, we deduce that
\begin{align*}
\|h-h_1\|_{L^p(\Omega)}=&\|G^TG-G_1^TG_1\|_{L^p(\Omega)}\\
\leq& \|G_2\|_{L^{\frac{2p}{2-p}}(\Omega)}(\|G\|_{L^2(\Omega)}+2\|G_2\|_{L^2(\Omega)})\\
\leq& C\|\tau(u)\|_{L^p(\Omega)}(\|\nabla u\|_{L^2(\Omega)}+\|\tau(u)\|_{L^p(\Omega)})\\
\leq& C\|\tau(u)\|_{L^p(\Omega)}(1+\|\tau(u)\|_{L^p(\Omega}).
\end{align*}
On the other hand, we have
\beno
\|h\|_{L^1(\Omega)}\leq \|\frac{\partial u}{\partial \bar{z}}\|_{L^2(\Omega)}^2\leq E(u,\Omega)\leq 1.
\eeno
This concludes that
\begin{align*}
\min\{\|h-h_1\|_{L^1(\Omega)},\ \|h\|_{L^1(\Omega)}\}\leq&\min\{C\|\tau(u)\|_{L^p(\Omega)}(1+\|\tau(u)\|_{L^p(\Omega)}),\ 1\}\\
\leq& C\|\tau(u)\|_{L^p(\Omega)}.
\end{align*}
Thus, the lemma is true for either $h_0=h_1$ or $h_0=0$.
\end{proof}

\begin{Lemma}\label{lem:app-small2}
There exist $\epsilon_0>0$ so that if
\beno
E(u,D_1\backslash D_{r})\leq \epsilon_0 \quad\text{for some }\, 0<r\leq \frac{1}{4},\quad \|\tau(u)\|_{L^p(D_1)}\le 1,
\eeno
and there exists a holomorphic function $h_{0,2r}$ in $D_{2r}$ satisfying
\beno
\|h-h_{0,2r}\|_{L^1(D_{2r})}+r^{\frac{2p-2}{p}}\|\tau(u)\|_{L^p(D_1)}\leq A_0.
\eeno
Then there exists a holomorphic function $h_0$ in $D_{1}$ such that
\beno
\|h-h_0\|_{L^1(D_{1})}\leq C\Big(A_0\ln\frac{1}{r}+\min\big\{\frac{A_0}{r},A_0^{\frac{1}{2}}+r^{\frac{1}{2}}\big\}+\|\tau(u)\|_{L^p(D_1)}\Big).
\eeno
Here $C$ is a constant independent of $A_0$ and $r$.
\end{Lemma}

\begin{proof}
Using the same notations as in Lemma \ref{lem:app-small}
with $\Omega=D_1\backslash D_{r}$, we have
\ben\label{eq:h-h1}
\|h-h_1\|_{L^1(D_1\backslash{D_r})}\leq C\|h-h_1\|_{L^p(D_1\backslash{D_r})}\leq C\|\tau(u)\|_{L^p(D_1)},
\een
 We denote by  $\sum\limits_{n\in \Z}a_nz^n (a_n\in \mathbb{C}^k)$ the Laurent expansion of $G_1(z)$ in $D_1\backslash D_{r}$.
Then we have
\beno
h_1(z)=\sum_{n\in \Z}b_nz^n\quad\text{with}\quad b_n=\sum_{m\in \Z}\langle a_m, a_{n-m}\rangle,
\eeno
and we define
\beno
{h}_0(z)=\sum_{n=0}^{\infty}b_nz^n.
\eeno
Hence, $h_0(z)$ is holomorphic in $D_1$. Since $h_{0,2r}$ is holomorphic in $D_{2r}$, we may write
\beno
h_{0,2r}(z)=\sum_{n=0}^{\infty}{b_n'}z^n\quad \text{in  }D_{2r}.
\eeno

Let $r_1=\f 54 r,\ r_2=\f 32 r, r_3=\f 74 r$. Then we obtain
\begin{align*}
\|h-{h}_0\|_{L^1(D_1)}\leq& \|h-{h}_0\|_{L^1(D_1\backslash{D_{r_2}})}+\|h-{h}_0\|_{L^1({D_{r_2}})}\\
\leq& \|h-{h}_1\|_{L^1(D_1\backslash{D_{r_2}})}+\|h-{h}_{0,2r}\|_{L^1({D_{r_2}})}\\
&+\|h_1-{h}_0\|_{L^1(D_1\backslash{D_{r_2}})}
+\|h_0-{h}_{0,2r}\|_{L^1({D_{r_2}})}\\
\leq& \|h-{h}_1\|_{L^1(D_1\backslash{D_{r}})}+\|h-{h}_{0,2r}\|_{L^1({D_{2r}})}\\
&+\sum\limits_{n=1}^{\infty}|b_{-n}|\|z^{-n}\|_{L^1(D_1\backslash{D_{r_2}})}
+\sum\limits_{n=0}^{\infty}|b_n-b_n'|\|z^{n}\|_{L^1({D_{r_2}})}\\
\leq& \|h-{h}_1\|_{L^1(D_1\backslash{D_{r}})}+\|h-{h}_{0,2r}\|_{L^1({D_{2r}})}\\
&+C\Big(|b_{-1}|+|b_{-2}|\ln\frac{1}{r_2}+\sum\limits_{n=3}^{\infty}|b_{-n}|\frac{1}{r_2^{n-2}}
+\sum\limits_{n=0}^{\infty}|b_n-b_n'|r_2^{n+2}\Big).
\end{align*}

{\bf Estimate of $b_{-n}$.}\medskip

Thanks to the assumption, we get
\begin{align}
\|h_{0,2r}-h_1\|_{L^1(D_{2r}\backslash{D_r})}\leq& \|h-h_{0,2r}\|_{L^1(D_{2r}\backslash{D_r})}+Cr^{\frac{2p-2}{p}}\|h-h_1\|_{L^p(D_{2r}\backslash{D_r})}\nonumber\\
\leq& CA_0,\nonumber
\end{align}
which implies that for $j=1,2,3$,
\beno
\int_{|z|=r_j}|z||h_{0,2r}-h_1||dz|\leq C\|h_{0,2r}-h_1\|_{L^1(D_{2r}\backslash{D_r})}\leq CA_0.
\eeno
Hence, we deduce that for $n\ge 1$,
\beno
|b_{-n}|=\frac{1}{2\pi}\Big|\int_{|z|=r_1}z^{n-1}h_1dz\Big|\leq r_1^{n-2}\int_{|z|=r_1}|z||h_{0,2r}-h_1||dz|\leq Cr_1^{n-2}A_0,
\eeno
and for $n\ge 0$,
\beno
|b_{n}-b_{n}'|=\frac{1}{2\pi}\Big|\int_{|z|=r_3}z^{-n-1}h_1dz\Big|\leq
r_3^{-n-2}\int_{|z|=r_3}|z||h_{0,2r}-h_1||dz|\leq Cr_3^{-n-2}A_0.
\eeno

{\bf Refined estimate of $b_{-1}$.}\medskip

Recalling $b_n=\sum\limits_{m\in \Z}\langle a_m,a_{n-m}\rangle$, we get
\ben
&&|b_{-1}|\leq 2\Big(|a_{-1}||a_0|+\sum_{n=1}^{\infty}|a_n||a_{-1-n}\Big),\label{eq:b-1}\\
&&|\langle a_{-1},a_{-1}\rangle|\leq |b_{-2}| +2\sum_{n=0}^{\infty}|a_n||a_{-2-n}|.\label{eq:a-1}
\een
A direct calculation yields that For $r\leq \rho_2< \rho_1\leq1$, we have
\beno
&&\|z^n\|_{L^2(D_{\rho_1}\backslash{D_{\rho_2}})}^2=\pi\dfrac{\rho_1^{2n+2}-\rho_2^{2n+2}}{n+1}\quad \text{for } n\neq -1,\\
&&\|z^{-1}\|_{L^2(D_{\rho_1}\backslash{D_{\rho_2}})}^2=2\pi\ln\dfrac{\rho_1}{\rho_2},\\
&&\|z^n\|_{L^2(D_{\rho_1}\backslash{D_{\rho_2}})}^2\leq\rho_1^{2}\|z^n\|_{L^2(D_{1}\backslash{D_{r}})}^2\quad \text{for } n\ge 0,\\
&&\|z^n\|_{L^2(D_{\rho_1}\backslash{D_{\rho_2}})}^2\leq({r}/{\rho_2})^{2}\|z^n\|_{L^2(D_{1}\backslash{D_{r}})}^2\quad \text{for } n\le -2.
\eeno
Then, using the fact that
\beno
\|G_1\|_{L^2({D_{\rho_1}\backslash{D_{\rho_2}}})}^2=\sum\limits_{n\in \Z} |a_n|^2\|z^n\|_{L^2(D_{\rho_1}\backslash{D_{\rho_2}})}^2,
\eeno
we infer that
\ben\label{eq:G1-up}
\|G_1\|_{L^2({D_{\rho_1}\backslash{D_{\rho_2}}})}^2\leq2\pi|a_{-1}|^2\ln\dfrac{\rho_1}{\rho_2}
+\max\big(\rho_1,\frac{r}{\rho_2}\big)^2\|G_1\|_{L^2(D_{1}\backslash{D_{r}})}^2.
\een

Using the fact that $\ln\|z^n\|_{L^2(D_{1}\backslash{D_{r}})}^2$ is a convex function of $n$, we obtain
\begin{align*}
\|z^n\|_{L^2(D_{1}\backslash{D_{r}})}\|z^{-1-n}\|_{L^2(D_{1}\backslash{D_{r}})}
&\geq\|z^1\|_{L^2(D_{1}\backslash{D_{r}})}\|z^{-2}\|_{L^2(D_{1}\backslash{D_{r}})}\\
&=\pi\sqrt{\dfrac{(1-r^{4})(1-r^2)}{2r^2}}
\end{align*}
for $n\ge 1$, and for $n\ge 0$,
\begin{align*}
\|z^n\|_{L^2(D_{1}\backslash{D_{r}})}\|z^{-2-n}\|_{L^2(D_{1}\backslash{D_{r}})}
&\geq\|z^0\|_{L^2(D_{1}\backslash{D_{r}})}\|z^{-2}\|_{L^2(D_{1}\backslash{D_{r}})}\\
&=\pi\dfrac{1-r^2}{r}.
\end{align*}
Then we conclude that
\begin{align*}
\|G_1\|_{L^2({D_{1}\backslash{D_{r}}})}^2=&\sum\limits_{n\in \Z} |a_n|^2\|z^n\|_{L^2(D_{1}\backslash{D_{r}})}^2\\
\geq& |a_0|^2\|z^0\|_{L^2(D_{1}\backslash{D_{r}})}^2+2\sum\limits_{n=1}^{\infty} |a_n||a_{-1-n}|\|z^n\|_{L^2(D_{1}\backslash{D_{r}})}\|z^{-1-n}\|_{L^2(D_{1}\backslash{D_{r}})}\\
\geq& |a_0|^2\pi(1-r^2)+2\sum\limits_{n=1}^{\infty} |a_n||a_{-1-n}|\pi\sqrt{\dfrac{(1-r^{4})(1-r^2)}{2r^2}}\\
\geq& |a_0|^2+\frac{2}{r}\sum\limits_{n=1}^{\infty} |a_n||a_{-1-n}|,
\end{align*}
and
\begin{align*}
\|G_1\|_{L^2({D_{1}\backslash{D_{r}}})}^2=&\sum\limits_{n\in \Z} |a_n|^2\|z^n\|_{L^2(D_{1}\backslash{D_{r}})}^2\\
\geq& 2\sum\limits_{n=0}^{\infty} |a_n||a_{-2-n}|\|z^n\|_{L^2(D_{1}\backslash{D_{r}})}\|z^{-2-n}\|_{L^2(D_{1}\backslash{D_{r}})}\\
\geq& 2\sum\limits_{n=1}^{\infty} |a_n||a_{-2-n}|\pi\dfrac{1-r^2}{r}\\
\geq& \frac{2}{r}\sum\limits_{n=1}^{\infty} |a_n||a_{-2-n}|,
\end{align*}
which along with (\ref{eq:b-1}) and (\ref{eq:a-1})  yield that
\ben
&&|b_{-1}|\leq 2|a_{-1}|\|G_1\|_{L^2({D_{1}\backslash{D_{r}}})}+r\|G_1\|_{L^2({D_{1}\backslash{D_{r}}})}^2,\label{eq:b-1-2}\\
&&|\langle a_{-1},a_{-1}\rangle|\leq |b_{-2}| +r\|G_1\|_{L^2({D_{1}\backslash{D_{r}}})}^2.\label{eq:a-1-2}
\een

It remains to estimate $a_{-1}$. For this, we denote
\beno
q(z)=q(|z|)=\dfrac{1}{2\pi}\int_{0}^{2\pi}R(ze^{i\theta})d\theta,\quad u=\left(\begin{array}{l}u\\0\end{array}\right)\in \R^{2k}.
\eeno
Let $\Om=D_1\backslash{D_{r}}$. We have
\begin{align*}
\left\|\dfrac{R-q}{\bar{z}}\right\|_{L^2(\Omega)}^2=&\int_r^1\int_{0}^{2\pi}\dfrac{|R(te^{i\theta})-q(t)|^2}{t^2}tdtd\theta
\\ \leq&\int_r^1\int_{0}^{2\pi}\dfrac{|\partial_{\theta}R(te^{i\theta})|^2}{t^2}tdtd\theta\\
\leq&\left\|\nabla{R}\right\|_{L^2(\Omega)}^2\leq C\epsilon_0.
\end{align*}
Moreover, we also have
\beno
A=R^T\dfrac{\partial u}{\partial z},\quad G=B^TA=B^TR^T\dfrac{\partial u}{\partial z},
\eeno
which gives
\begin{align*}
G_1=&G-G_2=B^TR^T\dfrac{\partial u}{\partial z}-G_2\\
=&q^T\dfrac{\partial u}{\partial z}+(R-q)^T\dfrac{\partial u}{\partial z}+(B-I_k)^TR^T\dfrac{\partial u}{\partial z}-G_2.
\end{align*}
Therefore for $ \rho=\sqrt{r/2}$, we have
\beno
(2\pi\ln2)a_{-1}=\int_{D_{2\rho}\backslash D_{\rho}}\frac{G_1(z)}{\bar{z}}dz\triangleq I_1+I_2,
\eeno
where
\beno
&&I_1=\int_{D_{2\rho}\backslash D_{\rho}}\dfrac{q^T}{\bar{z}}\dfrac{\partial u}{\partial z}dz,\\
&&I_2=\int_{D_{2\rho}\backslash D_{\rho}}\Big(\dfrac{(R-q)^T}{\bar{z}}\dfrac{\partial u}{\partial z}+(B-I_k)^T\dfrac{R^T}{\bar{z}}\dfrac{\partial u}{\partial z}-\dfrac{G_2}{\bar{z}}\Big)dz.
\eeno
Notice that
\begin{align*}
I_1=&\int_{\rho}^{2\rho}\int_{0}^{2\pi}\dfrac{q(t)^T}{te^{-i\theta}}\frac{e^{-i\theta}}{2}(\partial_t-i\frac1t\partial_{\theta})
u(te^{i\theta})tdtd\theta\\
=&\int_{\rho}^{2\rho}\int_{0}^{2\pi}\dfrac{q(t)^T}{2}\partial_tu(te^{i\theta})dtd\theta\in\R^{k}.
\end{align*}
Thus, we obtain
\begin{align*}
(2\pi\ln2)|\textrm{Im}a_{-1}|\leq|I_2|\leq& \Big\|\dfrac{R-q}{\bar{z}}\Big\|_{L^2(\Omega)}
\Big\|\dfrac{\partial u}{\partial z}\Big\|_{L^2(D_{2\rho}\backslash D_{\rho})}
+\Big\|\dfrac{1}{\bar{ z}}\Big\|_{L^2(D_{2\rho}\backslash D_{\rho})}\|G_2\|_{L^2(D_{2\rho}\backslash D_{\rho})}\\
&+\|B-I_k\|_{L^\infty(\Omega)}\Big\|\dfrac{1}{\bar{ z}}\Big\|_{L^2(D_{2\rho}\backslash D_{\rho})}\Big\|\dfrac{\partial u}{\partial z}\Big\|_{L^2(D_{2\rho}\backslash D_{\rho})}\\
\leq& C\epsilon_0^{\frac{1}{2}}\left\|\dfrac{\partial u}{\partial z}\right\|_{L^2(D_{2\rho}\backslash D_{\rho})}+C\epsilon_0\left\|\dfrac{\partial u}{\partial z}\right\|_{L^2(D_{2\rho}\backslash D_{\rho})}+C\left\|G_2\right\|_{L^2(D_{2\rho}\backslash D_{\rho})}\\
\leq& C\epsilon_0^{\frac{1}{2}}\left\|G\right\|_{L^2(D_{2\rho}\backslash D_{\rho})}+C\left\|G_2\right\|_{L^2(D_{2\rho}\backslash D_{\rho})}\\
\leq& C\epsilon_0^{\frac{1}{2}}\left\|G_1\right\|_{L^2(D_{2\rho}\backslash D_{\rho})}+C\left\|G_2\right\|_{L^2(D_{2\rho}\backslash D_{\rho})}\\
\leq& C\epsilon_0^{\frac{1}{2}}\Big({2\pi\ln2}|a_{-1}|^2
+(2\rho)^2\|G_1\|_{L^2(D_{1}\backslash{D_{r}})}^2\Big)^{\frac{1}{2}}+C\left\|G_2\right\|_{L^2(D_{2\rho}\backslash D_{\rho})},
\end{align*}
where we used (\ref{eq:G1-up}) in the last inequality. This along with (\ref{eq:a-1-2}) gives
\begin{align*}
|a_{-1}|^2=&\textrm{Re}\langle a_{-1},a_{-1}\rangle+2|\textrm{Im}a_{-1}|^2\\
\leq& |b_{-2}| +r\|G_1\|_{L^2({D_{1}\backslash{D_{r}}})}^2+C\epsilon_0\Big(|a_{-1}|^2
+r\|G_1\|_{L^2(D_{1}\backslash{D_{r}})}^2\Big)+C\left\|G_2\right\|_{L^2(D_{2\rho}\backslash D_{\rho})}^2\\
\leq& |b_{-2}| +Cr\|G_1\|_{L^2({D_{1}\backslash{D_{r}}})}^2+C\left\|G_2\right\|_{L^2(D_{2\rho}\backslash D_{\rho})}^2+C\epsilon_0|a_{-1}|^2.
\end{align*}
Taking $\epsilon_0$ small such that $C\epsilon_0\leq \frac{1}{2}$, we obtain
\beno
|a_{-1}|^2\leq C\big(|b_{-2}| +r\|G_1\|_{L^2({D_{1}\backslash{D_{r}}})}^2+\left\|G_2\right\|_{L^2(D_{2\rho}\backslash D_{\rho})}^2\big).
\eeno

Using the facts that
\begin{align*}
\|G_1\|_{L^2({D_{1}\backslash{D_{r}}})}\leq&\|G\|_{L^2({D_{1}\backslash{D_{r}}})}+\|G_2\|_{L^2({D_{1}\backslash{D_{r}}})}\\
\leq& C\|\nabla u\|_{L^2({D_{1}\backslash{D_{r}}})}+C\|\tau\|_{L^p({D_{1}\backslash{D_{r}}})}\leq C,
\end{align*}
and $|b_{-2}|\le CA_0$ and
\begin{align*}
\|G_2\|_{L^2(D_{2\rho}\backslash D_{\rho})}^2\leq& C\rho^{\frac{4p-4}{p}}\|G_2\|_{L^\frac{2p}{2-p}(D_{1}\backslash D_{r})}^2\leq C\rho^{\frac{4p-4}{p}}\|\tau\|_{L^p(D_{1}\backslash D_{r})}^2\\
\leq&  Cr^{\frac{2p-2}{p}}\|\tau\|_{L^p(D_{1}\backslash D_{r})}\leq CA_0,
\end{align*}
we deduce that
\ben
&&|a_{-1}|^2\leq C(A_0 +r),\label{eq:a-1-up}\\
&&|b_{-1}|\leq C(|a_{-1}| +r)\leq C\big(A_0^{\frac{1}{2}}+r^{\frac{1}{2}}\big).\label{eq:b-1-up}
\een
On the other hand, we also have $|b_{-1}|\leq C\dfrac{A_0}{r}$, hence,
\beno
|b_{-1}|\leq C\min\Big(\dfrac{A_0}{r},A_0^{\frac{1}{2}}+r^{\frac{1}{2}}\Big).
\eeno

Collecting the estimates of $b_n$,  we finally conclude that
\begin{align*}
\|h-{h}_0\|_{L^1(D_1)}\leq& C\Big(\|\tau(u)\|_{L^p(D_1)}+A_0\ln\dfrac{1}{r}+|b_{-1}|\Big)\\
 \leq& C\Big(\|\tau(u)\|_{L^p(D_1)}+A_0\ln\dfrac{1}{r}+\min\Big(\dfrac{A_0}{r},A_0^{\frac{1}{2}}+r^{\frac{1}{2}}\Big)\Big).
\end{align*}
This completes the proof of Lemma \ref{lem:app-small2}.
\end{proof}

\begin{Lemma}\label{lem:holo-exten}
Let $0<r<\rho$. If $h\in L^1(D(z_0,2r+\rho))$ and for any $z\in D(z_0,r+\rho)$,
there exists a holomorphic function $h_{z,r}$ in $D(z,r)$ so that
\beno
\|h-h_{z,r}\|_{L^1(D(z,r))}\leq A_0,
\eeno
then there exists a holomorphic function $h_0$ in $D(z_0,\rho)$ so that
\beno
\|h-h_0\|_{L^1(D(z_0,\rho))}\leq C\frac{\rho^3}{r^3}A_0.
\eeno
Here $C$ is a constant independent of $r,\rho,$ and $A_0$.
\end{Lemma}

\begin{proof}
Let $\phi$ be a radial cut-off function satisfying
\beno
\textrm{supp}\,\phi\subseteq D_{\frac{1}{2}},\quad \int_{\R^2}\phi dx=1.
\eeno
Let $\phi_{(r)}(x)=r^{-2}\phi\left(\frac{x}{r}\right)$ and
\beno
h_{(r)}(z)=\phi_{(r)}*h(z)=\int_{D(z_0,2r+\rho)}\phi_{(r)}(z-y)h(y)dy
\eeno
for $z\in D(z_0,r+\rho)$. We have $\phi_{(r)}*h_{z,r}=h_{z,r}$ in $D(z,\frac{r}{2})$ due to the mean value equality of holomorphic function.
Therefore,
\begin{align*}
\|h-h_{(r)}\|_{L^1(D(z,\frac{r}{2}))}\leq& \|h-h_{z,r}\|_{L^1(D(z,\frac{r}{2}))}+\|h_{(r)}-h_{z,r}\|_{L^1(D(z,\frac{r}{2}))}\\
=&\|h-h_{z,r}\|_{L^1(D(z,\frac{r}{2}))}+\|\phi_{(r)}*(h-h_{z,r})\|_{L^1(D(z,\frac{r}{2}))}\\
\leq&\|h-h_{z,r}\|_{L^1(D(z,r))}+\|\phi_{(r)}\|_{L^1}\|(h-h_{z,r})\|_{L^1(D(z,r))}\\
=&2A_0.
\end{align*}
Using Fubini theorem, we get
\begin{align*}
\int_{D(z_0,\frac{r}{2}+\rho)}\|h-h_{(r)}\|_{L^1(D(z,\frac{r}{2}))}dz
=&\int_{D(z_0,\frac{r}{2})}\|h-h_{(r)}\|_{L^1(D(z,\frac{r}{2}+\rho))}dz\\ \geq&\int_{D(z_0,\frac{r}{2})}\|h-h_{(r)}\|_{L^1(D(z_0,\rho))}dz\nonumber\\
=&\pi\left(\frac{r}{2}\right)^2\|h-h_{(r)}\|_{L^1(D(z_0,\rho))},
\end{align*}
which implies
\begin{align}
\|h-h_{(r)}\|_{L^1(D(z_0,\rho))}\leq&\Big(1+\frac{2\rho}{r}\Big)^2\sup_{z\in D(z_0,\frac{r}{2}+\rho)}\|h-h_{(r)}\|_{L^1(D(z,\frac{r}{2}))}\nonumber\\
\leq& \Big(1+\frac{2\rho}{r}\Big)^2(2A_0).\label{eq:h-hr}
\end{align}
Notice that in $D(z,r/2)$, we have
\begin{align*}
\frac{\partial}{\partial\bar{z}}h_{(r)}=\frac{\partial}{\partial\bar{z}}(h_{(r)}-h_{z,r})
=&\frac{\partial}{\partial\bar{z}}(\phi_{(r)}*(h-h_{z,r}))\\
=&\Big(\frac{\partial}{\partial\bar{z}}\phi_{(r)}\Big)*(h-h_{z,r}).
\end{align*}
Therefore,
\beno
\Big\|\frac{\partial}{\partial\bar{z}}h_{(r)}\Big\|_{L^1(D(z,\frac{r}{2}))}\leq
\Big\|\frac{\partial}{\partial\bar{z}}\phi_{(r)}\Big\|_{L^1}
\|h-h_{z,r}\|_{L^1(D(z,r))}\leq\frac{C}{r}A_0,
\eeno
which implies (similar to (\ref{eq:h-hr}))
\begin{align*}
\Big\|\frac{\partial}{\partial\bar{z}}h_{(r)}\Big\|_{L^1(D(z_0,\rho))}\leq&
\Big(1+\frac{2\rho}{r}\Big)^2\sup_{z\in D(z_0,\frac{r}{2}+\rho)}\Big\|\frac{\partial}{\partial\bar{z}}h_{(r)}\Big\|_{L^1(D(z,\frac{r}{2}))}\\
\leq& \Big(1+\frac{2\rho}{r}\Big)^2\frac{C}{r}A_0.
\end{align*}

Now we let $h_1=(\frac{1}{\pi z})*(\frac{\partial}{\partial\bar{z}}h_{(r)}\chi_{D(z_0,\rho)})$ and $h_0=h_{(r)}-h_1$.
Then $h_0$ is holomorphic in $D(z_0,\rho)$ and by (\ref{eq:h-hr}),
\begin{align*}
\|h-h_0\|_{L^1(D(z_0,\rho))}=&\|h-h_{(r)}+h_1\|_{L^1(D(z_0,\rho))}\\
\leq&\|h-h_{(r)}\|_{L^1(D(z_0,\rho))}+\|h_1\|_{L^1(D(z_0,\rho))}\\
\leq&\Big(1+\frac{2\rho}{r}\Big)^2\Big(2A_0+\frac{C\rho}{r}A_0\Big)\\
\leq& C\frac{\rho^3}{r^3}A_0,
\end{align*}
here we used
\beno
\|h_1\|_{L^1(D(z_0,\rho))}\leq\Big\|\frac{1}{\pi z}\Big\|_{L^1(D_{2\rho})}\Big\|\frac{\partial}{\partial\bar{z}}h_{(r)}\Big\|_{L^1(D(z_0,\rho))}
\leq\Big(1+\frac{2\rho}{r}\Big)^2\frac{C\rho }{r}A_0.
\eeno
This completes the proof of Lemma \ref{lem:holo-exten}.
\end{proof}
\medskip

Now we are in a position to prove Proposition \ref{prop:holomorphic approx}.\medskip

\no{\bf Proof of Proposition \ref{prop:holomorphic approx}}. We use the induction argument. The case of $m=1$ follows from Lemma \ref{lem:app-small}. Let us assume that the case of $m-1$ is true. The proof of the assertion for $m$ is split it into many cases.\medskip

{\bf Case 1.}\,$\|\tau(u)\|_{L^p(D_1)}\geq 1$. \medskip

In this case, we may take $h_0=0$, since
\beno
\|h\|_{L^1(D_1)}\leq\Big\|\frac{\partial u}{\partial z}\Big\|_{L^2(D_1)}^2\leq E(u,D_1)\leq m\epsilon_0\leq m\epsilon_0\|\tau(u)\|_{L^p(D_{1})}^{3^{1-m}}.
\eeno

{\bf Case 2.}\,$\|\tau(u)\|_{L^p(D_1)}\leq 1$ and $E(u,D_{\frac{1}{4}})\leq \epsilon_0$.\medskip

We first consider the function $v(z)=u(\frac{z}{4})$, which satisfies
\beno
E(v,D_1)=E(u,D_{\frac{1}{4}})\leq \epsilon_0,\quad \tau(v)(z)=\frac{1}{16}\tau(u)(\frac{z}{4}),
\eeno
hence,
\beno
\|\tau(v)\|_{L^p(D_{1})}=(\frac{1}{4})^{\frac{2p-2}{p}}\|\tau(u)\|_{L^p(D_{\frac14})}\leq\|\tau(u)\|_{L^p(D_1)}\leq 1.
\eeno
Then Lemma \ref{lem:app-small} ensures that there exists a holomorphic function $\widetilde{h}_0$ in $D_{1}$ so that
\beno
\|\widetilde{h}-\widetilde{h}_0\|_{L^1(D_{1})}\leq C\|\tau(v)\|_{L^p(D_{1})},
\eeno
where $\widetilde{h}(z)=\langle\frac{\partial v}{\partial z},\frac{\partial v}{\partial z}\rangle=\frac{1}{16}h(\frac{z}{4}).$
Therefore, $h_0(z)=16\widetilde{h}_0(4z)$ is holomorphic in $D_{\frac14}$ and satisfies
\begin{align*}
\|{h}-{h_0}\|_{L^1(D_{\frac14})}&=\|\widetilde{h}-\widetilde{h}_0\|_{L^1(D_{1})}\\
&\leq C\|\tau(v)\|_{L^p(D_{1})}\leq C\|\tau(u)\|_{L^p(D_{1})}\leq C\|\tau(u)\|_{L^p(D_{1})}^{3^{1-m}}.
\end{align*}

{\bf Case 3.}\,$\|\tau(u)\|_{L^p(D_1)}\leq 1$ and $Q(\frac{1}{8})\leq (m-1)\epsilon_0$, where
\beno
Q(t)\eqdef\sup\limits_{D(z,t)\subseteq D_1}E(u,D(z,t))\quad \text{for } t\in [0,1].
\eeno
Obvious,  $Q(t)$ is continuous and non-decreasing in $t$ and $Q(0)=0$. In this case, we let $v(z)=u(z'+\frac{z}{8})$ for $z'\in D_{1/2}$. Then $v$ satisfies
\beno
E(v,D_1)=E(u,D(z',\frac{1}{8}))\leq Q(\frac{1}{8})\leq (m-1)\epsilon_0,
\eeno
and $\tau(v)(z)=\frac{1}{64}\tau(u)(z'+\frac{z}{8})$, hence,
\beno
\|\tau(v)\|_{L^p(D_{1})}=(\frac{1}{8})^{\frac{2p-2}{p}}\|\tau(u)\|_{L^p(D(z',\frac{1}{8}))}\leq\|\tau(u)\|_{L^p(D_1)}.
\eeno
Then the induction assumption ensures that there exists a holomorphic function $\widetilde{h}_0$ in $D_{\frac14}$ so that
\beno
\|\widetilde{h}-\widetilde{h}_0\|_{L^1(D_{\frac14})}\leq C_{m-1}\|\tau(v)\|_{L^p(D_{1})}^{3^{2-m}}.
\eeno
where $\widetilde{h}=\langle\frac{\partial v}{\partial z},\frac{\partial v}{\partial z}\rangle=\frac{1}{64}h(z'+\frac{z}{8})$.
Set $r=\frac{1}{32}$. Then $h_{z',r}(z)=64\widetilde{h}_0(8(z-z'))$ is holomorphic in $D(z',\f 1 {32})$ and satisfies
\beno
\|{h}-{h_{z',r}}\|_{L^1(D(z',\f 1 {32}))}=\|\widetilde{h}-\widetilde{h}_0\|_{L^1(D_{\frac14})}\leq C_{m-1}\|\tau(v)\|_{L^p(D_{1})}^{3^{2-m}}\leq C_{m-1}\|\tau(u)\|_{L^p(D_{1})}^{3^{2-m}}.
\eeno
Therefore, $h$ satisfies the conditions in Lemma \ref{lem:holo-exten} with $z_0=0, r=\f 1 {32}, \rho=\f 14$ and $A_0=C_{m-1}\|\tau(u)\|_{L^p(D_{1})}^{3^{2-m}}$.
Thus, Lemma \ref{lem:holo-exten} ensures the existence of a holomorphic function ${h_0}$ in $D_{\frac14}$ so that
\beno
\|h-h_0\|_{L^1(D_{\frac14})}\leq C\frac{\rho^3}{r^3}A_0=8^3CC_{m-1}\|\tau(u)\|_{L^p(D_{1})}^{3^{2-m}}\leq CC_{m-1}\|\tau(u)\|_{L^p(D_{1})}^{3^{1-m}}.
\eeno

{\bf Case 4.}\,$\|\tau(u)\|_{L^p(D_1)}\leq 1,\ Q(\frac{1}{8})> (m-1)\epsilon_0$ and $E(u,D_{\frac{1}{4}})>\epsilon_0$.

In this case, there exists $0<r_0<\frac{1}{8}$ such that $Q(r_0)=(m-1)\epsilon_0$. Thus, there exists $D(z_0,r_0)\subseteq D_1$ so that $E(u,D(z_0,r_0))=(m-1)\epsilon_0$.
Hence,
\beno
E(u,D(z_0,r_0))+E(u,D_{\frac{1}{4}})>E(u,D_1),
\eeno
which implies that $D(z_0,r_0)\cap D_{\frac{1}{4}}\neq \emptyset,$ thus $|z_0|\leq\frac{1}{4}+r_0$ and then $D(z_0,4r_0)\subseteq D_1.$
For $z'\in D(z_0,3r_0),$ the function $v(z)=u(z'+r_0{z})$ satisfies
\beno
E(v,D_1)=E(u,D(z',r_0))\leq Q(r_0)= (m-1)\epsilon_0,
\eeno
and $\tau(v)(z)=r_0^2\tau(u)(z'+r_0{z})$, hence,
\beno
\|\tau(v)\|_{L^p(D_{1})}=r_0^{\frac{2p-2}{p}}\|\tau(u)\|_{L^p(D(z',r_0))}\leq
r_0^{\frac{2p-2}{p}}\|\tau(u)\|_{L^p(D_1)}.
\eeno
Then following argument of Case 3(using Lemma \ref{lem:holo-exten} with $A_0=C_{m-1}\Big(r_0^{\frac{2p-2}{p}}\|\tau(u)\|_{L^p(D_{1})}\Big)^{3^{2-m}},\\ r=\frac{r_0}{4}, \rho=2r_0$ ), we can conclude the existence of a holomorphic function $\widetilde{h}_0$ in $D(z_0,2r_0)$ such that
\beno
\|h-\widetilde{h}_0\|_{L^1(D(z_0,2r_0))}\leq C\frac{\rho^3}{r^3}A_0=8^3CC_{m-1}\Big(r_0^{\frac{2p-2}{p}}\|\tau(u)\|_{L^p(D_{1})}\Big)^{3^{2-m}}.
\eeno

Now set $\rho_0=\frac{1}{2}+r_0,\ r=r_0/\rho_0$ and consider the function $v(z)=u(\rho_0{z}+z_0)$. Then we have
\beno
&&E(v,D_1\backslash D_r)=E\left(u,D\left(z_0,\rho_0\right)\backslash D(z_0,r_0)\right)\leq E(u,D_1)-E(u,D(z_0,r_0))\leq \epsilon_0,\\
&&\|\tau(v)\|_{L^p(D_{1})}=\rho_0^{\frac{2p-2}{p}}\|\tau(u)\|_{L^p\left(D\left(z',\rho_0\right)\right)}\leq
\|\tau(u)\|_{L^p(D_1)}\leq 1.
\eeno
For $0<r_0<r<\frac{1}{4},$ the function $h_{0,2r}(z)=\rho_0^2\widetilde{h}_0(\rho_0{z}+z_0)$ is holomorphic in $D_{2r}$
and satisfies
\begin{align*}
\|\widetilde{h}-h_{0,2r}\|_{L^1(D_{2r})}+r^{\frac{2p-2}{p}}\|\tau(v)\|_{L^p(D_1)}=&
\|h-\widetilde{h}_0\|_{L^1(D(z_0,2r_0))}+r^{\frac{2p-2}{p}}\|\tau(v)\|_{L^p(D_1)} \\
 \leq& 8^3CC_{m-1}\left(r_0^{\frac{2p-2}{p}}\|\tau(u)\|_{L^p(D_{1})}\right)^{3^{2-m}}+r^{\frac{2p-2}{p}}\|\tau(u)\|_{L^p(D_1)}\\
 \leq& CC_{m-1}\left(r^{\frac{2p-2}{p}}\|\tau(u)\|_{L^p(D_{1})}\right)^{3^{2-m}},
\end{align*}
here $\widetilde{h}=\langle\frac{\partial v}{\partial z},\frac{\partial v}{\partial z}\rangle=\rho_0^2h(\rho_0{z}+z_0). $
Therefore, $v$ and $\widetilde{h}$ satisfies the conditions of Lemma \ref{lem:app-small2} with
\ben\label{eq:A0}
A_0=CC_{m-1}\Big(r^{\frac{2p-2}{p}}\|\tau(u)\|_{L^p(D_{1})}\Big)^{3^{2-m}}.
\een
Thus, there exists a holomorphic function $\widehat{h}_0$ in $D_{1}$ such that
\beno
\|\widetilde{h}-\widehat{h}_0\|_{L^1(D_{1})}\leq C\Big(A_0\ln\frac{1}{r}+\min\left\{\frac{A_0}{r},A_0^{\frac{1}{2}}+r^{\frac{1}{2}}\right\}+\|\tau(u)\|_{L^p(D_1)}\Big).
\eeno
For $A_0$ given by (\ref{eq:A0}), we have
\beno
&&A_0\ln\frac{1}{r}\leq C\frac{p}{2p-2}3^{m-2}C_{m-1}\|\tau(u)\|_{L^p(D_{1})}^{3^{2-m}}\leq
C3^{m}C_{m-1}\|\tau(u)\|_{L^p(D_{1})}^{3^{1-m}},\\
&&\min\Big(\frac{A_0}{r},A_0^{\frac{1}{2}}+r^{\frac{1}{2}}\Big)\leq A_0^{\frac{1}{2}}+\min\Big(\frac{A_0}{r},r^{\frac{1}{2}}\Big)\leq A_0^{\frac{1}{2}}+A_0^{\frac{1}{3}}\leq CC_{m-1}\|\tau(u)\|_{L^p(D_{1})}^{3^{1-m}}.
\eeno
This gives
\beno
\|\widetilde{h}-\widehat{h}_0\|_{L^1(D_{1})}\leq C3^{m}C_{m-1}\|\tau(u)\|_{L^p(D_{1})}^{3^{1-m}}.
\eeno
Let $h_{0}(z)=\rho_0^{-2}\widehat{h}_0(\frac{z-z_0}{\rho_0})$. Then $h_0(z)$ is holomorphic in $D\left(z_0,\rho_0\right)$
and satisfies
\begin{align*}
\|{h}-{h_0}\|_{L^1(D_{\frac{1}{4}})}\leq& \|{h}-{h_0}\|_{L^1(D\left(z_0,\rho_0\right))}\\
&=\|\widetilde{h}-\widehat{h}_0\|_{L^1(D_{1})} \leq C3^{m}C_{m-1}\|\tau(u)\|_{L^p(D_{1})}^{3^{1-m}}.
\end{align*}

Therefore,  the assertion holds for $m$ with $C_m=C3^{m}C_{m-1}$.  The proof of Proposition \ref{prop:holomorphic approx} is completed. 

\setcounter{equation}{0}

\section{Proof of Theorem \ref{thm:energy}}

Let us first prove the following energy decay estimates for the map 
satisfying the assumptions of Lemma \ref{lem:app-small2}.

\begin{Lemma}\label{lem:decay}
Let $u$ be as in Lemma \ref{lem:app-small2}. Then it holds that
\beno
&&E(u,D_{\rho}\backslash D_{r/\rho})\leq C\Big((A_0 +r)\ln\dfrac{\rho^2}{r}
+\rho^{\frac{4p-4}{p}}\Big)\quad \text{for}\quad \sqrt{r}\le \rho\le 1,\\
&&\textrm{osc}(u,D_{\rho}\backslash D_{r/\rho})\leq C\Big((A_0 +r)^{\frac{1}{2}}\ln\dfrac{1}{r}
+\rho^{\frac{2p-2}{p}}\Big)\quad\text{for}\quad 2\sqrt{r}\le \rho\le \f12.
\eeno
\end{Lemma}
\begin{proof}
Here we will use the notations in the proof of Lemma \ref{lem:app-small2}. By (\ref{eq:G1-up}) and (\ref{eq:a-1-up}), we get
\begin{align*}
E(u,D_{\rho}\backslash D_{r/\rho})\leq& C\|G\|_{L^2(D_{\rho}\backslash D_{r/\rho})}\leq C\Big(\|G_1\|_{L^2(D_{\rho}\backslash D_{r/\rho})}^2+\|G_2\|_{L^2(D_{\rho}\backslash D_{r/\rho})}^2\Big)\\
\leq& C\Big(|a_{-1}|^2\ln\dfrac{\rho^2}{r}+\rho^2\|G_1\|_{L^2(D_{1}\backslash{D_{r}})}^2+\rho^{\frac{4p-4}{p}}\|G_2\|_{L^\frac{2p}{2-p}(D_{1}\backslash D_{r})}^2\Big)\\
\leq& C\Big((A_0 +r)\ln\dfrac{\rho^2}{r}
+\rho^2+\rho^{\frac{4p-4}{p}}\|\tau\|_{L^p(D_{1}\backslash D_{r})}\Big)\\
\leq& C\Big((A_0 +r)\ln\dfrac{\rho^2}{r}
+\rho^{\frac{4p-4}{p}}\Big).
\end{align*}

If $2\sqrt{r}\leq\rho\leq\frac{1}{2}$, there exists a positive integer $\ell>1$ such that $e^{-\ell}\rho< r/\rho\leq e^{-\ell+1}\rho$ and $\ell\leq \ln\dfrac{\rho^2}{r}+1$.
Let $\rho_j=e^{-j}\rho$ for $0\leq j<\ell$ and $\rho_\ell=r/\rho$. Then by (\ref{eq:G1-up}) and Sobolev embedding and the fact that $G_1$ is holomorphic, we deduce that for $0< j\leq \ell$,
\begin{align*}
\textrm{osc}(u,D_{\rho_{j-1}}\backslash D_{\rho_j})\leq&  C\rho_j^{\frac{2p-2}{p}}
\|\nabla u\|_{L^\frac{2p}{2-p}(D_{e\rho_j}\backslash D_{\rho_j})}\leq C\rho_j^{\frac{2p-2}{p}}\|G\|_{L^\frac{2p}{2-p}(D_{e\rho_j}\backslash D_{\rho_j})}\\
\leq & C\rho_j^{\frac{2p-2}{p}}\Big(\|G_1\|_{L^\frac{2p}{2-p}(D_{e\rho_j}\backslash D_{\rho_j})}+\|G_2\|_{L^\frac{2p}{2-p}(D_{e\rho_j}\backslash D_{\rho_j})}\Big)\\
\leq & C\Big(\|G_1\|_{L^2(D_{2e\rho_j}\backslash D_{\rho_j/2})}+\rho_j^{\frac{2p-2}{p}}\|G_2\|_{L^\frac{2p}{2-p}(D_{1}\backslash D_{r})}\Big)\\
\leq& C\Big(|a_{-1}|+\max\{{2e\rho_j},2r/\rho_j\}\|G_1\|_{L^2(D_{1}\backslash D_{r})}+\rho_j^{\frac{2p-2}{p}}\|\tau\|_{L^p(D_{1}\backslash D_{r})}\Big)\\
\leq &C\Big(|a_{-1}|+\max\{{2e\rho_j},2r/\rho_j\}+\rho_j^{\frac{2p-2}{p}}\Big),
\end{align*}
which gives
\begin{align*}
\textrm{osc}(u,D_{\rho}\backslash D_{r/\rho})\leq&\sum_{j=1}^\ell \textrm{osc}(u,D_{\rho_{j-1}}
\backslash D_{\rho_j})\\
\leq& C\sum_{j=1}^\ell \Big(|a_{-1}|+\max\big({2e\rho_j},2r/\rho_j\big)+\rho_j^{\frac{2p-2}{p}}\Big)\\
\leq& C \Big(|a_{-1}|\ell+\rho+\rho^{\frac{2p-2}{p}}\Big)\\
\leq& C\Big((A_0 +r)^{\frac{1}{2}}\ln\dfrac{1}{r}
+\rho^{\frac{2p-2}{p}}\Big).
\end{align*}
The proof is finished.
\end{proof}
\medskip

Now let us complete the proof of Theorem \ref{thm:energy}.
By the construction of bubble tree,  it is sufficient to prove (\ref{eq:energy-lim}) and (\ref{eq:neckless}) under the assumption of $E(u_n,D(x_n^j,r_n^{f(j)}{\delta_0})\backslash D(x_n^j,r_n^{j}{R_0}))\leq\epsilon.$

We can also assume $\delta_0^{\frac{2p-2}{p}}\Lambda\leq 1,\ r_n^i\leq 1$. Let $m$ be a positive integer so that $m>\Lambda^2/\epsilon$.
Let $r=R_0r_n^i/(\delta_0r_n^{f(i)})$ and consider the function $w_n^i(z)=u_n(x_n^i+\delta_0r_n^{f(i)}{z})$.  It holds that
\beno
&&E(w_n^i,D_1\backslash D_r)=
E(u_n,D(x_n^i,r_n^{f(i)}{\delta_0})\backslash D(x_n^i,r_n^{i}{R_0}))\leq \epsilon,\\
&&\|\tau(w_n^i)\|_{L^p(D_{1})}=(\delta_0r_n^{f(i)})^{\frac{2p-2}{p}}
\|\tau(u_n)\|_{L^p\left(D(x_n^i,r_n^{f(i)}{\delta_0})\right)}\leq
\delta_0^{\frac{2p-2}{p}}\Lambda\leq 1.\eeno
For $n$ sufficiently large, $0<r_n^i<r<\frac{1}{8}$. We consider the function $v_n^i(z)=w_n^i(8r{z})=u_n(x_n^i+8R_0r_n^iz)$, which satisfies
\beno
&&E(v_n^i,D_1)=E\left(u_n,D(x_n^i,8r_n^{i}{R_0})\right)\leq \Lambda^2\leq m\epsilon,\\
&&\|\tau(v_n^i)\|_{L^p(D_{1})}=(8r)^{\frac{2p-2}{p}}\|\tau(w_n^i)\|_{L^p\left(D_{8r}\right)}\leq
(8r)^{\frac{2p-2}{p}}\leq 1.
\eeno
Thus, Proposition \ref{prop:holomorphic approx} and scaling argument ensure that there exists a holomorphic function $h_{0,2r}^{n,i}$ in $D_{2r}$ such that
\beno
\|{h}^{n,i}-h_{0,2r}^{n,i}\|_{L^1(D_{2r})}\leq C_{m}\|\tau(v_n^i)\|_{L^p(D_{1})}^{3^{1-m}}\leq C_{m}(8r)^{\frac{2p-2}{p}3^{1-m}}.
\eeno
Hence, we get
\begin{align*}
\|{h}^{n,i}-h_{0,2r}^{n,i}\|_{L^1(D_{2r})}
+r^{\frac{2p-2}{p}}\|\tau(w_n^i)\|_{L^p(D_1)}=&
C_{m}(8r)^{\frac{2p-2}{p}3^{1-m}}+r^{\frac{2p-2}{p}} \\
\leq& CC_{m}r^{\frac{2p-2}{p}3^{1-m}},
\end{align*}
here ${h}^{n,i}=\big\langle\frac{\partial w_n^i}{\partial z},\frac{\partial w_n^i}{\partial z}\big\rangle=(\delta_0r_n^{f(i)})^2h_n(x_n^i+\delta_0r_n^{f(i)}{z}).$
Thus, the conditions in Lemma \ref{lem:app-small2} are satisfied for $w_n^i$ with  $A_0=CC_{m}r^{\frac{2p-2}{p}3^{1-m}}$. For this $A_0$, it holds that
\beno
&&\lim\limits_{n\to \infty}(A_0 +r)\ln\dfrac{\rho^2}{r}=\lim\limits_{n\to \infty}(CC_{m}r^{\frac{2p-2}{p}3^{1-m}} +r)\ln\dfrac{\rho^2}{r}=0,\\
&&\lim\limits_{n\to \infty}(A_0 +r)^\frac{1}{2}\ln\dfrac{1}{r}=\lim\limits_{n\to \infty}(CC_{m}r^{\frac{2p-2}{p}3^{1-m}} +r)^\frac{1}{2}\ln\dfrac{1}{r}=0,
\eeno
as $r_n^i\to 0$(recall $r=R_0r_n^i/(\delta_0r_n^{f(i)})\to 0$). Then we apply Lemma \ref{lem:decay} to conclude that
\beno
&&\limsup\limits_{n\to \infty}E(u_n,D(x_n^i,\rho r_n^{f(i)}{\delta_0})\backslash D(x_n^i,r_n^{i}{R_0}/\rho))\leq C
\rho^{\frac{4p-4}{p}},\\
&&\limsup\limits_{n\to \infty}\ \textrm{osc}(u_n,D(x_n^i,\rho r_n^{f(i)}{\delta_0})\backslash D(x_n^i,r_n^{i}{R_0}/\rho))\leq C
\rho^{\frac{2p-2}{p}},
\eeno
which yield (\ref{eq:energy-lim}) and (\ref{eq:neckless}) by taking $\rho\rightarrow 0$. The proof of Theorem \ref{thm:energy} is completed.\endproof

\section{A necessary and sufficient condition of energy identity}

Let us first recall the following result \cite{DT, LZ2}.

\begin{Lemma}\label{lem:the tangential energy}
If $\tau(u_n)$ is bounded in $L^p$ for some $p>1$, then the tangential energy on the neck domain is vanishing, i.e.,
\beno
\lim_{\delta\rightarrow 0}\lim_{R\rightarrow \infty}\lim_{n\rightarrow \infty}\int_{D(0,r_n^{f(i)}\delta)\setminus D(0,r_n^iR)} |x|^{-2}|\partial_\theta u_n(x^i_n+x)|^2dx=0.
\eeno
\end{Lemma}

Notice that
\begin{align*}
h_n(x^i_n+x)=&h_n(x^i_n+re^{i\theta})\\
=&\frac14e^{-2i\theta}\Big(|\partial_r u_n|^2-r^{-2}|\partial_\theta u_n|^2-\frac{2i}{r}\langle\partial_r u_n,\partial_\theta u_n\rangle\Big)(x^i_n+re^{i\theta}),
\end{align*}
which motivates the following equivalent statement of the energy identity.

\begin{Proposition}\label{energy eqv}
Let $h_n=\big\langle\frac{\partial u_n}{\partial z},\frac{\partial u_n}{\partial z}\big\rangle$. Then the energy identity holds if and only if
\ben\label{eq1}
\lim_{\delta\rightarrow 0}\limsup_{n\rightarrow \infty}\|h_n\|_{L^1(D_{\delta})}=0.
\een
\end{Proposition}

\begin{proof}
On the one hand, if (\ref{eq1})  holds,  then we have by Lemma \ref{lem:the tangential energy} that
\begin{align*}
&\lim\limits_{\delta\rightarrow 0}\lim\limits_{R\rightarrow \infty}\limsup\limits_{n\rightarrow \infty}
E\big(u_n, D(x^i_n,r_n^{f(i)}\delta)\setminus D(x^i_n,r_n^iR)\big)\\
&=\lim\limits_{\delta\rightarrow 0}\lim\limits_{R\rightarrow \infty}\limsup\limits_{n\rightarrow \infty}\dfrac{1}{2}
\int_{D(0,r_n^{f(i)}\delta)\setminus D(0,r_n^iR)}\left(|\partial_r u_n|^2-|x|^{-2}|\partial_\theta u_n|^2\right)(x^i_n+x)dx
\\ &\leq\lim\limits_{\delta\rightarrow 0}\lim\limits_{R\rightarrow \infty}\limsup\limits_{n\rightarrow \infty}
2\int_{D(0,r_n^{f(i)}\delta)\setminus D(0,r_n^iR)}|h_n(x^i_n+x)|dx\\
&=\lim\limits_{\delta\rightarrow 0}\lim\limits_{R\rightarrow \infty}\limsup\limits_{n\rightarrow \infty}
2\|h_n\|_{L^1(D(x^i_n,r_n^{f(i)}\delta)\setminus D(x^i_n,r_n^iR))}\\
&\leq\lim\limits_{\delta\rightarrow 0}\lim\limits_{R\rightarrow \infty}\limsup\limits_{n\rightarrow \infty}
2\|h_n\|_{L^1(D_{\delta})}=0,
\end{align*}
which gives (\ref{eq:energy-lim}), thus the energy identity.\smallskip

On the other hand, if  the energy identity holds, we denote
$$w_n^i(z)=u_n(x_n^i+r_n^{i}{z}),\ \ {h}_n^{i}=\big\langle\frac{\partial w_n^i}{\partial z},\frac{\partial w_n^i}{\partial z}\big\rangle,\ \ {h}^{i}=\big\langle\frac{\partial w^i}{\partial z},\frac{\partial w^i}{\partial z}\big\rangle$$
for $i=1,\cdots k_0.$ Then ${h}_n^{i}(z)=(r_n^{i})^2h_n(x_n^i+r_n^{i}{z})$, and ${h}^{i}$ is a $L^1$ holomorphic function in $\mathbb{C}$, thus ${h}^{i}=0$.
Thanks to
\beno
&&u_n^i\to w^i\ \text{strongly in}\ W^{1,2}(D_R\setminus\cup_{x\in Z_i}D(x,\delta))\ \text{as}\ n\to \infty,
\eeno
we infer that as $n\to \infty$,
\beno
&&h_n^i\to h^i=0\quad \text{strongly in}\ L^{1}(D_R\setminus\cup_{x\in Z_i}D(x,\delta)),\\
&&\|h_n\|_{L^1(D(x_n^i,r_n^iR)\setminus\cup_{x\in Z_i}D(x_n^i+r_n^ix,r_n^i\delta))}=\|h_n^i\|_{L^1(D_R\setminus\cup_{x\in Z_i}D(x,\delta))}\to 0.
\eeno
Notice that 
\begin{align*}
\|h_n\|_{L^1(D_{\delta})}&\leq\sum\limits_{i=1}^{k_0}\|h_n\|_{L^1(D(x_n^i,r_n^iR)\setminus\cup_{x\in Z_i}D(x_n^i+r_n^ix,r_n^i\delta))}+\sum\limits_{i=1}^{k_0}\|h_n\|_{L^1(D(x^i_n,r_n^{f(i)}\delta)\setminus D(x^i_n,r_n^{i}R))}\\
&\le \sum\limits_{i=1}^{k_0}\|h_n\|_{L^1(D(x_n^i,r_n^iR)}\setminus\cup_{x\in Z_i}D(x_n^i+r_n^ix,r_n^i\delta))\\
&\quad+\sum\limits_{i=1}^{k_0}E(u_n,D(x^i_n,r_n^{f(i)}\delta)\setminus D(x^i_n,r_n^{i}R)).
\end{align*}
Thus, (\ref{eq1}) follows easily from (\ref{eq:energy-lim}).
\end{proof}
\medskip

Using Lemma \ref{energy eqv} and Proposition \ref{prop:holomorphic approx},  let us present another proof of the energy identity.\smallskip

Let $m$ be a positive integer so that $m>\Lambda^2/\epsilon$. Then for fixed $0<\delta_1<1,$ we consider the function $w_n(z)=u_n(\delta_1{z})$, which satisfies
\beno
&&E(w_n,D_1)=E\left(u_n,D_{\delta_1}\right)\leq \Lambda^2\leq m\epsilon,\\
&&\|\tau(w_n)\|_{L^p(D_{1})}=\delta_1^{\frac{2p-2}{p}}\|\tau(u_n)\|_{L^p\left(D_{\delta_1}\right)}\leq
\delta_1^{\frac{2p-2}{p}}\Lambda.
\eeno
Thus, Proposition \ref{prop:holomorphic approx} and scaling argument ensure that there exists a holomorphic function $h_{0,n}$ in $D_{\delta_1/4}$ such that
\beno
\|{h}_{n}-h_{0,n}\|_{L^1(D_{\delta_1/4})}\leq C_{m}\|\tau(w_n)\|_{L^p(D_{1})}^{3^{1-m}}\leq CC_{m}(\delta_1)^{\frac{2p-2}{p}3^{1-m}}.
\eeno
Therefore, for $0<\delta<\delta_1/4,$
\begin{align*}
\|{h}_{n}\|_{L^1(D_{\delta})}\leq& \|{h}_{n}-h_{0,n}\|_{L^1(D_{\delta})}+\|h_{0,n}\|_{L^1(D_{\delta})}\\ \leq&
\|{h}_{n}-h_{0,n}\|_{L^1(D_{\delta_1/4})}+\left({4\delta}/{\delta_1}\right)^2\|h_{0,n}\|_{L^1(D_{\delta_1/4})}\\ \leq&
2\|{h}_{n}-h_{0,n}\|_{L^1(D_{\delta_1/4})}+\left({4\delta}/{\delta_1}\right)^2\|h_{n}\|_{L^1(D_{\delta_1/4})}\\ \leq&
CC_{m}(\delta_1)^{\frac{2p-2}{p}3^{1-m}}+\left({4\delta}/{\delta_1}\right)^2\Lambda^2.
\end{align*}
Then (\ref{eq1}) follows by first letting $\delta\to 0$, then letting $\delta_1\to 0$.

\section*{Acknowledgements}

The authors thank Xiangrong Zhu for many profitable comments and discussions. Wendong Wang is partially supported by NSF of China under Grants 11301048 and the fundamental research funds for the central universities. Zhifei Zhang is partially supported by NSF of China under Grants 11371039 and 11425103.


\begin{thebibliography}{99}


\bibitem{DT}  W. Ding and G. Tian, {\it Energy identity for a class of approximate harmonic maps from surfaces}, Comm. Anal. Geom.,  3(1995), 543-554.

\bibitem{Helein} F. H\'{e}lein, {\it Harmonic Maps, Conservation Laws and Moving Frames}, Diderot, Paris, 1997.

\bibitem{LZ1} J. Li and X. Zhu, {\it Small energy compactness for approximate harmomic mappings}, Comm. Contemp. Math., 13(2011), 741-763.

\bibitem{LZ2} J. Li and X. Zhu, {\it Energy identity for the maps from a surface with tension field bounded in $L^p$}, Pacific J. Math., 260 (2012), 181-195.

\bibitem{LW-CVPDE} F.-H. Lin and C. Wang, {\it Energy identity of harmonic map flows from surfaces at finite singular time}, Calc. Var. Partial Differential Equations,  6(1998),  369-380.

\bibitem{LW-CAG} F.-H.  Lin and C.Wang, {\it Harmonic and quasi-harmonic spheres, II}, Comm. Anal. Geom., 10 (2002), 341-375.

\bibitem{LW-book} F.-H. Lin and C. Wang,  {\it The analysis of harmonic maps and their heat flows}, World Scientific Publishing Co. Pte. Ltd., Hackensack, NJ, 2008.

\bibitem{Luo} Y. Luo, {\it Energy identity and removable singularities of maps from a riemann surface with tension field unbounded in $L^2$}, Pacific J. Math., 256 (2012), 365-380.

\bibitem{Mcd}  D. McDuff and D. Salamon,  {\it J-holomorphic curves and symplectic topology}, American Mathematical Society Colloquium Publications, 52, American Mathematical Society, Providence, RI,  2012.

\bibitem{Parker} T. H. Parker, {\it Bubble tree convergence for harmonic maps}, J. Differential Geom., 44 (1996),  595-633.

\bibitem{Qing} J. Qing, {\it On singularities of the heat flow for harmonic maps from surfaces into spheres},  Comm. Anal. Geom., 3(1995), 297-315.

\bibitem{QT} J. Qing and G. Tian, {\it Bubbling of the heat flows for harmonic maps from surfaces}, Comm. Pure Appl. Math., 50(1997), 295-310.

\bibitem{SU} J. Sacks and K. Uhlenbeck, {\it The existence of minimal immersions of 2-spheres},  Ann. of Math., 113(1981), 1-24.

\bibitem{Top1} P. Topping, {\it Repulsion and quantization in almost-harmonic maps, and asymptotics of the harmonic map flow}, Ann. of Math., 159(2004), 465-534.

\bibitem{Top2} P. Topping, {\it Winding behaviour of finite-time singularities of the harmonic map heat flow}, Math. Z.,  247 (2004), 279-302.

\bibitem{Wang} C. Wang, {\it Bubble phenomena of certain Palais-Smale sequences from surfaces to general targets}, Houston J. Math.,  22(1996), 559-590.

\bibitem{Zhu1} X. Zhu, {\it No neck for approximate harmonic maps to the sphere}, Nonlinear Anal., 75(2012), 4339-4345.

\bibitem{Zhu2} X. Zhu, {\it Bubble tree for approximate harmonic maps}, Proc. Amer. Math. Soc. ,142 (2014),  2849-2857.




\end{thebibliography}
\end{document}